\documentclass[11pt,letterpaper]{amsart}
\usepackage[left=1in,right=1in,bottom=0.7in,top=0.8in]{geometry}
\usepackage{amsfonts}
\usepackage{amsmath,amssymb}
\usepackage[pdfpagelabels,hyperindex]{hyperref}
\usepackage{amsthm}
\usepackage{fancyhdr}

\newtheorem{theorem}{Theorem}[section]
\newtheorem{proposition}[theorem]{Proposition}
\newtheorem{lemma}[theorem]{Lemma}
\theoremstyle{remark}
\newtheorem{remark}[theorem]{Remark}

\makeatletter
\let\c@equation\c@theorem
\makeatother
\numberwithin{equation}{section}

\title{\textbf{4-cycle-free induced subgraphs of grid graphs}}
\author{Taiki Aiba \and Ernie Croot}
\date{}

\begin{document}

\begin{abstract}
In this paper we consider sets $S$ of vertices of the $n \times n$ grid graph $G_n$ where the induced subgraph $G_S$ having vertex set $S$ contains no $4$-cycles.  We first quickly establish the maximum size of any such induced subgraph.  Then, for even $n$ we estimate the number of $4$-cycle-free induced subgraphs of maximum and near-maximum size. If ${\mathcal F}_\varepsilon$ denotes the family of $4$-cycle-free induced subgraphs with at least $(3/4-\varepsilon)n^2$ vertices, then we show
$$
\log|{\mathcal F}_0|=(1+o(1))n\log n.
$$
Then, we also establish exponential upper and lower bounds for $|{\mathcal F}_\varepsilon|$. This estimate is proved using a local letter-pattern encoding of the unique missing vertex in a set of $2\times2$ blocks that tile the grid; the upper bound for $|{\mathcal F}_\varepsilon|$ uses a coarser block tiling (using larger blocks). We discuss the problem of closing the gap between our upper and lower bounds for $|{\mathcal F}_\varepsilon|$, and a possible defect-container approach.
\end{abstract}

\maketitle

\section{Introduction}

Let $G_n$ be the grid graph on $[n]^2 := \{1,2,...,n\}^2$, where the vertices $(x,y)$ and $(x',y')$ are connected if $|x-x'| + |y-y'| = 1$. For $S\subseteq [n]^2$, let $G_S$ denote the induced subgraph of $G_n$ having vertex set $S$. Every $4$-cycle in $G_n$ is the boundary of a unit square, so $G_S$ is $4$-cycle-free exactly when the complement of $S$ contains a vertex in every unit square of the grid (i.e. $4$ vertices forming a $4$-cycle of $G_n$). This simple local description makes the extremal problem easy but leaves a nontrivial counting problem: how many maximum or nearly maximum $4$-cycle-free induced subgraphs $G_S$ are there?

The problems considered here are related to work on large induced forests in grid graphs. Alon, Mubayi, and Thomas considered induced forests under degree constraints \cite{alon}, and Caragiannis, Kaklamanis, and Kanellopoulos determined asymptotically the minimum feedback vertex set in $d$-dimensional grids; in dimension two their result gives that a maximum induced-forest has vertex density $2/3$ \cite{caragiannis}. Here we forbid only $4$-cycles, rather than all cycles, and so the corresponding maximum density is $3/4$ for even $n$. Counting $4$-cycle-free graphs in other models also has a substantial literature, including the work of Kleitman and Winston \cite{kleitmanwinston}.  

It is worth mentioning that we had also considered the problems of counting the number of maximum and near-maximum sets $S$ where $G_S$ avoids $6$-cycles, $8$-cycles, and so on; and we found these to be significantly more difficult.  

We also considered applying hypergraph-container methods, such as those of Balogh-Morris-Samotij and Saxton-Thomason.  These provide a general framework for counting independent sets in forbidden-configuration problems \cite{baloghmorrissamotij,saxtonthomason}. In the present bounded-degree hypergraph, however, a direct container application does not recover the dependence on $\varepsilon$ suggested by the lower bound.

The main results are stated after the elementary extremal lemma below. Sections \ref{sec:max-count} and \ref{sec:near-count} prove the maximum and near-maximum counting estimates, respectively, and Section \ref{sec:discussion} discusses the problem of closing the gaps between the upper and lower bounds for the near-maximum case.

\section{Extremal size and main counting theorem}

Partition $[2\lfloor n/2\rfloor]^2$ into disjoint $2\times2$ blocks. We call these the \emph{standard blocks}.

\begin{lemma}\label{2x2lemma}
The largest $4$-cycle-free subset of $[n]^2$ has size
\[
n^2-\left\lfloor\frac n2\right\rfloor^2.
\]
In particular, for even $n$ the maximum size is $3n^2/4$.
\end{lemma}

\begin{proof}
Every standard block is itself a $4$-cycle, so a $4$-cycle-free set omits at least one vertex from each of the $\lfloor n/2\rfloor^2$ standard blocks. Hence
\[
|S|\le n^2-\left\lfloor\frac n2\right\rfloor^2.
\]
Equality is attained by deleting the vertices $(2i,2j)$ with $1\le i,j\le\lfloor n/2\rfloor$: every unit square contains exactly one such vertex. 
\end{proof}

For even $n$, let ${\mathcal F}$ be the family of all $S\subseteq[n]^2$ for which $G_S$ is $4$-cycle-free, and for $\varepsilon\geq 0$ put
\[
{\mathcal F}_\varepsilon=\{S\in{\mathcal F}:|S|\ge(3/4-\varepsilon)n^2\}.
\]
Thus ${\mathcal F}_0$ is the family of maximum $4$-cycle-free sets.

\begin{theorem}\label{thm:main}
Let $n$ be even.
\begin{enumerate}
\item[(1)]
\[
\left(\left\lfloor\frac n4\right\rfloor+1\right)^{n-1}
\le |{\mathcal F}_0|
\le \left(\frac n2+1\right)^n.
\]
In particular, $\log|{\mathcal F}_0|=(1+o(1))n\log n$ as $n\to\infty$ through even integers.

\item[(2)] If $\varepsilon>0$ is sufficiently small, then
\[
(c-o(1))\varepsilon\log(1/\varepsilon)n^2
\le \log|{\mathcal F}_\varepsilon|
\le (\kappa + o_{\varepsilon \to 0}(1))\varepsilon^{1/3}(\log(1/\varepsilon))^{2/3} n^2 + o_{n \to \infty}(n^2),
\]
where $c=\log(4/3)$ and $\kappa=4\cdot3^{-4/3}(\log 2)^{1/3}$.
\end{enumerate}
\end{theorem}

\begin{remark}
For odd $n$, Lemma \ref{2x2lemma} gives the exact extremal size, but the normalization $(3/4-\varepsilon)n^2$ is no longer centered at the maximum. We therefore state the counting theorem only for even $n$.
\end{remark}

\section{Counting maximum sets}\label{sec:max-count}

Assume $n=2m$. By Lemma \ref{2x2lemma}, every $S\in{\mathcal F}_0$ contains exactly three vertices in each standard block. Thus each block has a unique missing vertex. Encode its horizontal position by an $m\times m$ array $A$ with entries $L,R$, and its vertical position by an $m\times m$ array $B$ with entries $D,U$: in block $(i,j)$, $A_{i,j}=L$ or $R$ according as the missing vertex has first coordinate $2i+1$ or $2i+2$, while $B_{i,j}=D$ or $U$ according as its second coordinate is $2j+1$ or $2j+2$.

For example, when $n=4$, deleting $(2,2),(4,2),(1,4),(3,4)$ gives
\[
A=\begin{array}{cc} L&L\\ R&R\end{array},
\qquad
B=\begin{array}{cc} U&U\\ U&U\end{array}.
\]

The elementary restrictions are that $A$ has no horizontal occurrence of $LR$, and $B$ has no vertical occurrence of $U$ above $D$. These conditions exclude unit squares crossing the boundary of two horizontally or vertically adjacent standard blocks. A unit square can also meet four standard blocks, which gives two additional coupled restrictions.

\begin{lemma}[Letter-pattern criterion]\label{lem:criterion}
A pair $(A,B)$ represents a member of ${\mathcal F}_0$ if and only if $A$ contains no horizontal occurrence of $LR$, $B$ contains no vertical occurrence of $U$ above $D$, and no $2\times2$ position has either of the simultaneous pairs
\begin{equation}\label{complete-restrictions}
\left(
\begin{array}{cc}L&L\\ R&R\end{array},
\begin{array}{cc}D&U\\ D&U\end{array}
\right)
\qquad\text{or}\qquad
\left(
\begin{array}{cc}R&R\\ L&L\end{array},
\begin{array}{cc}U&D\\ U&D\end{array}
\right).
\end{equation}
\end{lemma}

\begin{proof}
A unit square lies either inside one standard block, across two horizontally adjacent blocks, across two vertically adjacent blocks, or across four standard blocks. The first case is already excluded because every standard block has one missing vertex. The next two cases give the restrictions on $A$ and $B$, and the last case gives exactly the two simultaneous patterns in \eqref{complete-restrictions}. These conditions are therefore necessary and sufficient.
\end{proof}

We first prove the upper bound in Theorem \ref{thm:main}(1). If a row of $A$ avoids $LR$, then it consists of a string of $R$'s followed by a string of $L$'s, so there are at most $m+1$ possibilities for each row and at most $(m+1)^m$ possibilities for $A$. Similarly there are at most $(m+1)^m$ possibilities for $B$. Hence
\[
|{\mathcal F}_0|\le(m+1)^{2m}=\left(\frac n2+1\right)^n.
\]

We next give a family attaining the correct logarithmic order.

\begin{proposition}\label{prop:F0-lower}
If $n=2m$, then
\[
|{\mathcal F}_0|\ \geq\
\left(\left\lfloor\frac m2\right\rfloor+1\right)^{2m-1}.
\]
Consequently, $\log|{\mathcal F}_0|=(1+o(1))n\log n$ as $n\to\infty$ through even integers.
\end{proposition}

\begin{proof}
Index the block rows from bottom to top and the block columns from left to right, and put $\ell=\lfloor m/2\rfloor$. Choose integers $a_1,\ldots,a_m$ and $b_1,\ldots,b_m$ subject to the following:
\begin{itemize}
    \item $\ell \leq a_i \leq m$, where $1 \leq i \leq \ell$, and $a_{\ell+1} = \ell$,
    \item $0 \leq a_i \leq \ell$, where $\ell+2 \leq i \leq m$,
    \item $0 \leq b_j \leq \ell$, where $1 \leq j \leq \ell$,
    \item $\ell \leq b_j \leq m$, where $\ell+1 \leq j \leq m$.
\end{itemize}
Define the $i$th block row of $A$ by
\[
A_{1,i}\cdots A_{m,i}\ =\ {\rm R}^{a_i}{\rm L}^{m-a_i},
\]
and define the $j$th column of $B$, read from bottom to top, by
\[
B_{j,1}\cdots B_{j,m}\ =\ {\rm U}^{b_j}{\rm D}^{m-b_j}.
\]
Every row of $A$ avoids ${\rm LR}$ and every column of $B$ avoids a ${\rm U}$ above a ${\rm D}$.

Basically, all the entries in the bottom left $\ell \times \ell$ block of $A$ are R, while all the entries in the top right are L; and then all the entries in the top left of $B$ are D, while all the entries in the bottom right are U.  The entries in $A$ in its other two blocks are a variable mix of R and L; and the other two blocks of $B$ are a variable mix of D and U.  These blocks are chosen in such a way as to prevent
 (\ref{complete-restrictions}) from occurring in any of the blocks, since in either of the two cases you need a mix of L and R and also D and U, not a solid block of all L's and R's or D's and U's.  

Now we check that this choice avoids (\ref{complete-restrictions}) more precisely:  the first simultaneous pair would require, for some adjacent block rows $i$ and $i+1$ and columns $j$ and $j+1$,
\[
a_i\ \geq\ j+1,\quad a_{i+1}\ \leq\ j-1,
\quad b_j\ \leq\ i-1,\quad b_{j+1}\ \geq\ i+1.
\]
The first two inequalities force either $i\leq\ell$ and $j\geq\ell+1$, or $i\geq\ell+1$ and $j\leq\ell-1$. In the first case the last two inequalities force $i\geq\ell+1$, and in the second they force $i\leq\ell-1$, a contradiction. For the second simultaneous pair all four inequalities are reversed; the same two cases give the same contradictions. Thus every chosen pair $(A,B)$ satisfies the necessary-and-sufficient criterion and determines a member of ${\mathcal F}_0$.

The number of choices is
\[
(m-\ell+1)^m(\ell+1)^{m-1}
\ \geq\ (\ell+1)^{2m-1}.
\]
Together with the upper bound $(m+1)^{2m}$, this yields
\[
\log|{\mathcal F}_0|\ =\ (1+o(1))\,2m\log(2m)
\ =\ (1+o(1))n\log n.
\]
\end{proof}

\section{Near-extremal counting}\label{sec:near-count}

We now prove Theorem \ref{thm:main}(2). The upper bound separates the grid into larger blocks and observes that a near-maximum set will be extremal on all but a small number of them.

\subsection{Upper bound} Suppose $H$ is an induced subgraph of the $n \times n$ grid and is $4$-cycle-free, where $H$ has at least $(3/4 - \varepsilon)n^2$ vertices. Let 
\begin{equation}\label{kchoice}
k\ =\ \left \lfloor \left ( \frac{1}{72\cdot \log 2}\right )^{1/3} \left (\frac{\log(1/\varepsilon)}{\varepsilon}\right )^{1/3} \right \rfloor, 
\end{equation}
and suppose $n$ is divisible by $2k$.
This divisibility assumption is harmless for the asymptotic estimate: for general even $n$, apply the argument to the largest subgrid whose side length is divisible by $2k$ and encode the remaining boundary strips in $2^{O(kn)}$ ways. For fixed $\varepsilon$ this factor is $\exp(o(n^2))$ as $n\to\infty$.
Then let $S_{u,v}$ denote the total number of vertices of $H$ that happen to belong to the $2k \times 2k$ subgrid with coordinates $(i,j)$ where 
\begin{equation}\label{ij}
2(u-1)k\ <\ i\ \leq\ 2uk \quad \text{and} \quad
2(v-1)k\ <\ j\ \leq\ 2vk.
\end{equation}
Then the total number of vertices of $H$ equals
$$
\sum_{u=1}^{n/2k} \sum_{v=1}^{n/2k} S_{u,v}\ \geq\ 
\left ( \frac{3}{4} - \varepsilon \right ) n^2.
$$

Now, since $H$ is $4$-cycle-free we have that $H$
can contain at most $3 k^2$ elements in each 
of the $2k \times 2k$ subgrids. Let $Z_0$ be
the pairs $(u,v)$ where $S_{u,v} = 3k^2$ and let 
$Z_1$ be the pairs $(u,v)$ where $S_{u,v} \leq 3k^2-1$.
Note that
$$
|Z_0| + |Z_1|\ =\ \frac{n^2}{4k^2}.
$$

We have now that
\begin{eqnarray*}
\frac{3 n^2}{4} - |Z_1|\ &=&\ 3k^2 |Z_0| + (3k^2 -1) |Z_1|\\
&\geq&\ \sum_{(u,v) \in Z_0} S_{u,v} + \sum_{(u,v) \in Z_1} 
S_{u,v}\\
&=&\ \sum_{u,v} S_{u,v}\\
&\geq&\ \left ( \frac{3}{4}
-\varepsilon \right ) n^2.
\end{eqnarray*}So, $|Z_1|\ \leq\ \varepsilon n^2$.

Let $Z \subset \{1,\ldots,n/2k\}^2$,
which one may think of as corresponding to a collection of $2k \times 2k$ subgrids of $G$ with coordinates $(i,j)$
as given by (\ref{ij}). Let $F_1(Z)$ be the total number of subsets of size at most $3 k^2 |Z|$ of all those subgrids, combined. Note that if we chose a subset any larger than this, then at least one of the $2k \times 2k$ subgrids would contain more than $3k^2$ elements, and therefore would have a $4$-cycle.

Clearly, 
\begin{equation}\label{F1}
F_1(Z)\ =\ \sum_{j \leq 3 k^2 |Z|} \binom{4 k^2 |Z|}{j}
\ <\ 2^{4k^2 |Z|}.
\end{equation}
(Since the sum contains the central binomial coefficient $\binom{4k^2 |Z| }{2k^2 |Z|}$, little is lost by using this upper bound.)

Also, let $F_0(Z)$ denote the total number of possible subsets of all the vertices in those $2k \times 2k$ subgrids where each $2 \times 2$ subgrid of {\it those} contains exactly $3$ vertices (and where the $2k\times 2k$ subgrid contains no $4$-cycles); so, each $2k \times 2k$ subgrid would contain $3 k^2$ vertices. As proved in the previous section, there are at most $(k+1)^{2k}$ possibilities for each $2k \times 2k$ subgrid assuming the subgraph is $4$-cycle-free, which means that we get the bound
\begin{equation}\label{F0}
F_0(Z)\ \leq\ ((k+1)^{2k})^{|Z|}\ =\ (k+1)^{2k|Z|}.
\end{equation}

It follows that the total number of induced subgraphs
$H$ that do not contain 4-cycles is at most
$$
I\ :=\ \sum_{\substack{Z_1 \subseteq \{1,\ldots,n/2k\}^2 \\ |Z_1| \leq \varepsilon n^2}} F_1(Z_1) F_0(Z_0),
$$
where here $Z_0$ simply denotes the complement of $Z_1$
inside $\{1,2,\ldots,n/2k\}^2$. Using (\ref{F1}) and (\ref{F0}) as above, we get that
\begin{eqnarray*}
|I|\ \leq\ \sum_{\substack{Z_1 \subseteq \{1,\ldots,n/2k\}^2 \\ |Z_1| \leq \varepsilon n^2}} 2^{4 k^2|Z_1|} (k+1)^{2k|Z_0|}\ &\leq&\ (k+1)^{n^2/2k} 2^{4 \varepsilon k^2 n^2} \sum_{\substack{Z_1 \subseteq \{1,\ldots,n/2k\}^2 \\ |Z_1| \leq \varepsilon n^2}} 1 \\
&\leq&\ (k+1)^{n^2/2k} 2^{4\varepsilon k^2 n^2} 
\sum_{0 \leq j \leq \varepsilon n^2} \binom{n^2/4k^2}{j}.
\end{eqnarray*} 
Now, if $k^2 < 1/16\varepsilon$, then this sum would not contain a central binomial coefficient, so
$$
|I|\ \ll\ (k+1)^{n^2/2k} 2^{4\varepsilon k^2n^2}
\binom{n^2 / 4k^2}{\varepsilon n^2}
\ \leq\ (k+1)^{n^2/2k} 2^{4\varepsilon k^2n^2} 
\left ( \frac{e}{4 \varepsilon k^2} \right )^{\varepsilon n^2}.
$$

From our choice of $k$ in
(\ref{kchoice}), this gives
$$
|I|\ \ll\ \exp \left ( (\kappa + o_{\varepsilon \to 0}(1)) \varepsilon^{1/3} (\log 1/\varepsilon)^{2/3}  n^2 + o_{n \to \infty}(n^2) \right ),
$$
where
\begin{equation}\label{kappa_def}
\kappa\ =\ 4 \cdot 3^{-4/3} \cdot (\log 2)^{1/3}.
\end{equation}

Note that this is significantly smaller than the total number of subsets of vertices of the $n \times n$ grid, which has $2^{n^2}$ such subsets. However, when $\varepsilon$ is sufficiently far away from $0$, then this upper bound starts growing on the order of $2^{c n^2}$. 



\subsection{Lower bound}

To prove a {\it lower} bound on the number of induced subgraphs of size at least $(3/4-\varepsilon) n^2$ vertices, we begin by letting $S \subseteq [n]^2$ be any set of vertices where $|S| = 3n^2/4$, and $G_S$ has no $4$-cycles. And then we simply count the number of subsets $S' \subseteq S$ with $|S'| \geq (3/4 - \varepsilon)n^2$. Like the subgraph $G_S$, the induced subgraphs $G_{S'}$ are also $4$-cycle-free. 

For $\varepsilon \leq 3/8$, the number of such $S'$ is at least 
$$
\sum_{k \leq \varepsilon n^2} 
\binom{3n^2/4}{k}\ \geq\ 
\binom{3n^2/4}{[\varepsilon n^2]}\ 
\geq\ \left ( \frac{3}{4 \varepsilon}
\right )^{\varepsilon n^2 - 1}
\ \geq\ \exp\left ( c \varepsilon \log(1/\varepsilon) n^2\right ),
$$
where we can take $c = \log(4/3)$.

\section{The remaining gap}\label{sec:discussion}

The lower and upper bounds in Theorem \ref{thm:main}(2) do not have the same dependence on $\varepsilon$.  One approach to closing the gap would be to use hypergraph methods, specifically ``container methods".  For example, consider the hypergraph ${\mathcal H}_n$ where the vertex set is $[n]^2$ and with $4$-uniform hyperedges given by all the $4$-cycles in the grid (i.e. all the unit squares).  An independent set in this hypergraph would be a set of vertices, where no $4$ of them completely fill a hyperedge; so, independent sets correspond to sets $S$ where $G_S$ is $4$-cycle-free, and so we can view ${\mathcal F}$ as the family of independent sets of ${\mathcal H}_n$. However, since ${\mathcal H}_n$ has bounded degree, the general container theorems \cite{baloghmorrissamotij,saxtonthomason} do not by themselves improve the $\varepsilon$-dependence obtained above.

There is, however, additional structure near the extremal threshold. Write $t=\varepsilon n^2$. Among the standard $2\times2$ blocks, at most $t$ can contain at most two vertices of a set $S\in{\mathcal F}_\varepsilon$; every other block contains exactly three vertices and therefore carries the $L/R$ and $U/D$ labels of Section \ref{sec:max-count}. This suggests a defect-sensitive container argument in which the exceptional blocks form a fingerprint and the letter patterns are encoded only away from that fingerprint. The present upper bound loses information by treating an entire $2k\times2k$ block as arbitrary as soon as it contains a defect. Controlling how the defects interrupt the monotone letter patterns could plausibly reduce the exponent toward the lower-bound scale $\varepsilon\log(1/\varepsilon)n^2$.

\section{Acknowledgements}

The authors used Anthropic's Fable model and GPT-5.6 to help with writing the manuscript.  Mostly this involved proofreading and  helping with the wording.  However, these models did also help a little with the math.   The only new mathematical contributions to the paper made by the models in this manuscript were the lower-bound for $|{\mathcal F}_0|$ in Theorem 2.2(1), and then also some of the discussion about using container methods.  We, ourselves, already tried container methods (and could not get them to work for our problems), and just asked the model if it could solve them using those methods.  I could not, at least in a simple way.

\end{document}